\documentclass[draft]{birkjour}
\usepackage{hyperref}
\usepackage{tikz}
\usepackage{pgfplots}

\usepackage[noadjust]{cite}
\usepackage{xcolor}
\RequirePackage[all]{xy}

\newtheorem{thm}{Theorem}[section]
\newtheorem{corollary}[thm]{Corollary}
\newtheorem{lemma}[thm]{Lemma}

\theoremstyle{definition}
\newtheorem{definition}[thm]{Definition}
\theoremstyle{remark}

    \newtheorem{example}{Example}
\numberwithin{equation}{section}

\newcommand{\BibTeX}{B\kern-0.1emi\kern-0.017emb\kern-0.15em\TeX}
\newcommand{\XYpic}{$\mathrm{X\kern-0.3em\raisebox{-0.18em}{Y}}$-$\mathrm{pic}\,$}

\newcommand{\cl}{C \kern -0.1em \ell}  

\newcommand{\ed}{\end{document}}

\begin{document}

%
%
%
%
%
%
%
%
%

\title[Inverse and Determinant in $\cl_{p,q}$ over Odd-dimensional Spaces]{Explicit Formula for Inverse and\\ Determinant in Geometric Algebras\\ over Odd-dimensional Vector Spaces}
\author[K. Abdulkhaev]{Kamron  Abdulkhaev}
\address{%
HSE University\\ 
101000 Moscow\\
Russia\\}
\email{ksabdulkhaev@edu.hse.ru}

\author[D. Shirokov]{Dmitry Shirokov}
\address{%
HSE University\\ 
101000 Moscow\\
Russia\\ \\
\and \\ \\
Institute for Information Transmission Problems of the Russian Academy of Sciences\\
127051 Moscow\\
Russia\\}
\email{dshirokov@hse.ru}
\date{\today}
\dedicatory{Last Revised:\\ \today}
\begin{abstract}
In this paper, we present  explicit formulas for the inverse and determinant in geometric (Clifford) algebras over vector spaces of dimension $n=7$. The derivation of these formulas is made possible by generalizing the concept of conjugation to basis conjugation operations. We further develop a general method for constructing such formulas over odd-dimensional spaces from the known even-dimensional case. To validate computational utility of the results, we provide a numerical implementation of the formulas. The code implementation is available at the repository \url{github.com/kamranuz/clifford_7d}. These formulas extend previous results for lower dimensions and offer new insights for applications in mathematical physics and computational geometry.
\end{abstract}
\label{page:firstblob}
\maketitle
\section{Introduction}

Geometric (Clifford) algebras provide a powerful algebraic framework that unifies various mathematical structures and finds extensive applications in mathematics, physics, and engineering \cite{Hestenes,Hitzer3,Lasenby,Lounesto}. Central to their utility is the ability to represent geometric transformations and encode multilinear algebraic operations in a coordinate-free manner. Among the fundamental algebraic invariants, the determinant plays a crucial role in characterizing invertibility.
While explicit determinant formulas are well-understood for geometric algebras of low dimensions, notably for $n \leq 6$ \cite{Acus,Hitzer,Shirokov} (see also Section 3), the case of $n=7$ presents unique challenges due to the increased algebraic complexity and the richer structure of the algebra \cite{Hitzer2}. Existing approaches often rely on additional algebra isomorphism constructions, which obscure the intrinsic geometric nature of the determinant and complicate direct computations within the algebra itself.

In this paper, we overcome these challenges by deriving an explicit formula that is independent of the representation for the inverse and determinant in geometric algebras $\cl_{p,q}$ with $p~+~q~=~7$. Our approach generalizes the classical notion of conjugation to a broader family of basis conjugation operations, enabling the expression of the determinant formula purely in terms of algebraic operations intrinsic to $\cl_{p,q}$ with $p+q=7$. This construction not only clarifies the algebraic structure underlying the determinant but also provides a practical computational tool for testing invertibility and calculating inverses of multivectors in geometric algebras. Beyond the specific result for $n=7$, we present a general method for constructing such determinant and inverse formulas for Clifford algebras over odd-dimensional vector spaces from known formulas for the lower even-dimensional case. 

We also provide a numerical implementation of the derived formulas with performance evaluation. The implementation enables direct computation of the determinant and inverse of any multivector in a  Clifford algebras over seven-dimensional vector spaces. Our performance benchmarks show that while the new formulas have a modest increase in computational cost, they offer substantially improved robustness and precision compared to previous methods. The implementation is available at the repository \url{github.com/kamranuz/clifford_7d}.

Our work extends and refines earlier determinant formulas obtained via algebra isomorphisms \cite{Hitzer2}, offering new insights and computational advantages. In contrast to known methods based on Faddeev-LeVerrier algorithm \cite{Shirokov, Prodanov}, our approach is not recursive and yields in explicit formulas. Such explicit formulas have interesting from a theoretical point of view. Given the importance of geometric algebras over  seven-dimensional vector spaces in mathematical physics, the explicit formulas presented here have potential applications in both theoretical investigations and numerical implementations~\cite{Hitzer3}.

This paper is an extended version of the short note in Conference Proceedings \cite{Inv7d}, where we presented the explicit isomorphism-free formula using basis conjugations, and its optimized version. In this extended version, we additionally provide a comprehensive numerical implementation and performance evaluation, and introduce a general recursive method to derive determinant formulas for Clifford algebras over odd-dimensional vector spaces directly from known even-dimensional cases.

The paper is organized as follows. In Section~2, we review the main definitions relevant to geometric algebra. In Section~3, we review known determinant formulas in geometric algebra for dimensions up to $n=6$. In Section~4, we focus on the Hitzer and Sangwine determinant formula for $n=7$ derived via algebra isomorphisms. In Section~5, we develop the explicit determinant formula independent of any algebra isomorphism by introducing basis conjugations. In Section~6, we present an optimized version of the formula. In Section~7, we provide a numerical implementation of the formulas and
evaluate their performance. In Section~8,  we present a method to derive such formulas in the for Clifford algebras over odd-dimensional vector spaces, using known formulas for lower even-dimensional vector spaces case. Finally, In Section~9, we offer conclusions and perspectives for further research.

\section{Main definitions}

Let us consider the (Clifford) geometric algebra $\cl_{p,q}$, $p+q=n$ \cite{Hestenes,Lasenby,Lounesto} with the generators $e_1$, $e_2$, \ldots, $e_n$ and the identity element $e$. The generators satisfy the conditions
\begin{eqnarray}
    e_ae_b+e_be_a=2\eta_{ab}e,\qquad a,b=1,\ldots,n, \nonumber
\end{eqnarray}
where $\eta =(\eta_{ab})={\rm diag}(1,\ldots, 1, -1, \ldots , -1)$ is the diagonal matrix with its first $p$ entries equal to $1$ and the last $q$ entries equal to $-1$ on the diagonal. 

The standard basis of $\cl_{p,q}$ is given by the set
\begin{eqnarray}
    \mathbb{B}_{p,q} := \{ e_{a_1a_2\cdots a_k} \mid a_1 < a_2 < \cdots < a_k,
    \quad k = 0, 1, \ldots, n\},\label{clbasis}
\end{eqnarray}
where each $e_{a_1a_2 \cdots a_k} = e_{a_1} e_{a_2} \cdots e_{a_k}$ denotes a product of generators with strictly increasing indices. This set $\mathbb{B}_{p,q}$ forms a basis of $\cl_{p,q}$. We call the subspace of $\cl_{p,q}$ of elements, which are linear combinations of the basis elements with multi-indices of length $k$, the subspace of grade $k$ and denote it by $\cl_{p,q}^k$. Elements of grade 0 are identified with scalars $\cl_{p,q}^0  \equiv \mathbb{R}, e \equiv 1$.
The projection of any element $U\in\cl_{p,q}$ onto the subspace $\cl^k_{p,q}$ is denoted by $\langle U\rangle_k$ in this paper.
We have
\begin{eqnarray}
    \langle U+V \rangle_k =\langle U \rangle_k+\langle V \rangle_k,\quad
    \langle\lambda U \rangle_k=\lambda \langle U \rangle_k,\quad
\lambda \in \mathbb{R},\quad
    U,V\in \cl_{p,q}.\label{proj}
\end{eqnarray}
An arbitrary element (multivector) $U\in\cl_{p,q}$ can be written in the form
\begin{eqnarray}
        U = \sum_{k=0}^n\langle U\rangle_k,\qquad
    \langle U\rangle_k \in \cl_{p,q}^k.\label{conjug}
\end{eqnarray}
The scalar $\langle U \rangle_0$ is called the scalar part of $U$. We have the property
\begin{eqnarray}
\langle UV \rangle_0=\langle VU \rangle_0,\qquad \forall U, V\in\cl_{p,q}.\label{comm}
\end{eqnarray}

\begin{definition}[\cite{Shirokov}]\label{def:conj}
\normalfont
Any operation of the form
\begin{eqnarray}
    U \mapsto
    \sum_{k=0}^n\lambda_k\langle U \rangle_k,\qquad
    \lambda_k = \pm 1\qquad
\end{eqnarray}
\normalfont
is called \textit{an operation of conjugation} in  $\cl_{p,q}$.
\end{definition}
Note that the operation of conjugation  is an involution, meaning that applying it twice yields the identity transformation.
The operations of conjugation commute with each other.
We have three classical operations of conjugation: the grade involution, the reversion, and the Clifford conjugation \footnote{The Clifford conjugation is a composition of the grade involution $\widehat{~~}$ and the reversion $\widetilde{~~}$. Note that some authors \cite{Lounesto} denote the Clifford conjugation by $\stackrel{\over\quad}{\quad}$. We do not use separate notation for the Clifford conjugation in this paper and write the combination of the two symbols $\widehat{~~}$ and $\widetilde{~~}$.}:
\begin{eqnarray}
    \widehat{U}=\sum_{k=0}^n (-1)^k\langle U \rangle_k,~~
    \widetilde{U}=\sum_{k=0}^n (-1)^\frac{k(k-1)}{2}\langle U \rangle_k,~~
    \widehat{\widetilde{U}}=\sum_{k=0}^n (-1)^\frac{k(k+1)}{2}\langle U \rangle_k.\label{conjug2}
\end{eqnarray}
These operations have the following properties
\begin{eqnarray}
\label{conj_prop}
    \widehat{UV}=\widehat{U}\widehat{V},\qquad
    \widetilde{UV}=\widetilde{V}\widetilde{U},\qquad
    \widehat{\widetilde{UV}}=\widehat{\widetilde{V}}\widehat{\widetilde{U}},\qquad
    \forall\, U,V \in\cl_{p,q}.\label{conjug2p}
\end{eqnarray}
\begin{definition}[\cite{Shirokov}]\label{def:det}
    Let $U\in\cl_{p,q}$ be an arbitrary multivector and $p+q=n$. Let us denote $N:=2^{[\frac{n+1}{2}]}$, where square brackets mean taking the integer part. We define a sequence of elements $U_{(k)}$ and scalars $C_{(k)}$ with $k=1,2,\ldots,N$ recursively as follows
    $$
       U_{(1)}:=U,\quad
       U_{(k+1)}:=U(U_{(k)}-C_{(k)}),\quad
       C_{(k)}:=\frac{N}{k}\langle U_{(k)}\rangle_0\in\cl_{p,q}^0.
    $$
    Then the determinant of an element $U\in\cl_{p,q}$ is defined as 
    $$
        {\rm Det}(U):= -C_{(N)}\in\cl_{p,q}^0.
    $$
\end{definition}
Unlike the straightforward definition of the determinant in geometric algebra given by ${\rm Det}(U):= {\rm det}(\beta(U))$  via matrix isomorphism $\beta$ from $\mathbb{C}\otimes\cl_{p,q}$ to certain matrix algebra of dimension $N$ (see for the details \cite{Shirokov}), the definition above is intrinsic to geometric algebra $\cl_{p,q}$. The proposed definition uses only addition operation, geometric multiplication, and projection operation. Nevertheless, it is equivalent to the straightforward definition of the determinant via the matrix determinant.

\begin{definition} \label{def:hyper}
The algebra of \textit{hyperbolic numbers} (also called \textit{split-complex numbers}) is defined as $\mathbb{S} := \{ z = x + j y \mid x, y \in \mathbb{R},\ j^2 = 1 \}$, where $j$ is the hyperbolic unit satisfying $j^2 = 1$, in contrast to the imaginary unit $i$ of complex numbers, where $i^2 = -1$. The set $\mathbb{S}$ forms a two-dimensional commutative algebra over the real numbers.
\end{definition}

\section{Known Determinant Formulas in Geometric Algebra for $n \leq 6$}

In this section, we provide a brief overview of known to date determinant formulas for multivectors in geometric algebras $\cl_{p,q}$ with $n = p + q \leq 6$. Determinant formulas are critical for characterizing invertibility and computing inverses in geometric algebras. 

We begin by introducing the necessary definitions and constructions. Following that, we summarize key results up to $n=6$ from \cite{Acus,Shirokov} in Theorem~\ref{nleq7}. 

For our convenience, we use the following notation for the operation of conjugation, introduced in \cite{Hitzer2}. Given  a multivector $U \in \cl_{p,q}$, we could represent operation of conjugation that changes signs of a subset of grades $M \subseteq {\{1,\ldots,n\}}$ in the following form 
    $$
        m_{M}(U) = U - 2\sum_{k\in M} \langle U \rangle_{k},
    $$
    where $\langle U\rangle_{k}$ denotes the projection of the element $U$ onto grade $k$.

\begin{example}
    For an element \( U \in \cl_{7,0} \) and the set of indices \( M = \{0,7\} \), we have
    $$
        m_{0,7}(U) = U - 2\langle U \rangle_{0} - 2\langle U \rangle_{7},
    $$
    where \( \langle U \rangle_{0} \) is the scalar part, and \( \langle U\rangle_{7} \) is the component of the highest grade.
\end{example}

Several results have been established concerning the computation of the determinant and inverse in geometric algebras. Explicit formulas for the inverse of Clifford algebra elements have been derived for dimensions $n\leq5$, as reported in \cite{Hitzer,Dadbeh,Shirokov0}, each employing distinct techniques. Later, an explicit formula for the $n = 6$ case was provided in \cite{Acus}. A summary of these findings can be found in~\cite{Shirokov}, and we present them here in Theorem \ref{nleq7}.
\begin{thm}[\cite{Shirokov}]\label{nleq7}
    For the cases $n\leq 6$, there exist the following determinant formulas~${\rm Det}:~\cl_{p,q}~\rightarrow~\cl_{p,q}^0$
    $$
        {\rm Det}(U)=UF(U)\in\cl_{p,q}^0,
    $$
    where
    {\small
    \begin{eqnarray}   
        F(U):=
        \begin{cases}
            \widehat{U},&\text{if $n=1$,}\\
            \widehat{\widetilde{U}},&\text{if $n=2$,}\\
            \widetilde{U}\widehat{U}\widehat{\widetilde{U}},&\text{if $n=3$,}\\
            \widetilde{U}m_{4}(\widehat{U}\widehat{\widetilde{U}}),&\text{if $n=4$,}\\
            \widetilde{U}m_{4,5}(\widehat{U}\widehat{\widetilde{U}})m_{4,5}(U\widetilde{U}m_{4,5}(\widehat{U}\widehat{\widetilde{U}})),&\text{if $n=5$,}\\
            \frac{1}{3}\widetilde{U}\widehat{U}\widehat{\widetilde{U}}m_{4,5,6}(\widehat{U}\widehat{\widetilde{U}}U\widetilde{U})&\\
            +\frac{2}{3}\widetilde{U}m_{4,5,6}\left(m_{4,5,6}(\widehat{U}\widehat{\widetilde{U}})m_{4,5,6}\left(m_{4,5,6}(\widehat{U}\widehat{\widetilde{U}})m_{4,5,6}(U\widetilde{U})\right)\right),&\text{if $n=6$.}
        \end{cases}\nonumber
    \end{eqnarray}}
    If  ${\rm Det}(U)\neq0$, then there exists $U^{-1}$ with the following explicit form
    $$
        U^{-1}=\frac{F(U)}{ {\rm Det}(U)}. 
    $$
\end{thm}

\section{Hitzer \& Sangwine Result for $n=7$}
Recently, Hitzer and Sangwine proposed the following determinant formula for $n=7$ using a technique based on algebra isomorphisms~\cite{Hitzer2}.

\begin{thm}[\cite{Hitzer2}]\label{n=7}
   Let $U\in\cl_{p,q}$ with $p+q=7$. Then the determinant of the element $U$ can be calculated by the formula
    \begin{eqnarray}
       {\rm Det}(U)&=&\phi^{-1}(By)\widehat{\phi^{-1}(By)}\in\cl_{p,q}^0,\quad
       W:=B\widetilde{B},\quad
       B:=\phi(U),\nonumber\nonumber\\
       y&:=&\widetilde{B}\left(
       \frac{1}{3}Wm_{1,4,5}(WW)+
       \frac{2}{3}m_4(m_4(W)m_{1,4,5}(m_4(W)m_4(W)))
       \right),\nonumber
    \end{eqnarray}
    where $\phi$ is an isomorphism
    \begin{eqnarray}
        \phi: \cl_{p,q}\xrightarrow{\sim}
        \begin{cases}
            \mathbb{C}\otimes\cl_{p',q'},\quad\text{if $q'$ is even,}\\
            \mathbb{S}\otimes\cl_{p',q'},\quad\text{else,}
        \end{cases}\quad
       p'+q'=6.\label{phi}
    \end{eqnarray}
\end{thm}
The explicit formula for the determinant presented above not only characterizes invertibility, but also provides a direct method for computing the inverse element in the algebra. In particular, if ${\rm Det}(U)\neq 0$, the inverse can be written as
$$
   U^{-1} = \frac{\phi^{-1}(y)\widehat{\phi^{-1}(By)}}{{\rm Det}(U)},
$$
where all terms are defined as in the theorem above.

\section{Explicit Formula of Determinant for $n=7$}
In this section we consider the determinant formula for elements of geometric algebra in the case $n=7$, independent of the isomorphism~(\ref{phi}).  First, however, we need to explicitly construct such an isomorphism $\phi$. 
\begin{definition}
\label{def:phi}
    Let $\{e_1,\ldots,e_7\}$ be the generators of $\cl_{p,q}$ with $p+q=7$ and $\{E_1,\ldots,E_6\}$ be the generators of $\cl_{p',q'}$ with $p'+q'=6$. Let us define the isomorphism 
    \begin{eqnarray}    
        &&\phi:\cl_{p,q}\xrightarrow{\sim}
        \begin{cases}
            \mathbb{C}\otimes\cl_{p',q'},\quad \text{if $q'$ is even},\nonumber\\
            \mathbb{S}\otimes\cl_{p',q'},\quad \text{else},
        \end{cases}   \nonumber\\
        &&\phi(e_{1~m+1}):=1\otimes E_m,\quad m=1,\ldots,6,\nonumber\\
        &&\phi(e_{1234567}):=r\otimes E,\quad r=\begin{cases}
            i\in\mathbb{C},\quad \text{if $q'$ is even},\\
            j\in\mathbb{S},\quad \text{else},
        \end{cases}\nonumber
    \end{eqnarray}
    where $E$ is the identity element of the algebra $\cl_{p',q'}$.
    \end{definition}
    \vspace{-0.5cm}
\begin{table}
\begin{center}
    \begin{tabular}{ |l|l|c|c| } 
     \hline
      $U\in\cl_{p,q}$ & $\pm\phi(U)$ & grade of $U$ & grade of $\phi(U)$ \\\hline\hline
     
     $e$ & $1\otimes e$ & 0 & 0 \\\hline 
     $e_{1\ i+1}$ & $1\otimes E_i$ & 2 & 1 \\\hline 
     $e_{i+1\ j+1}$ & $1\otimes E_{ij}$ & 2 & 2 \\ \hline
     $e_{1\ i+1\ j+1\ k+1}$ & $1\otimes E_{ijk}$ & 4 & 3 \\ \hline
     $e_{i+1\ j+1\ k+1\ l+1}$ & $1\otimes E_{ijkl}$ & 4 & 4 \\ \hline
     $e_{1\ i+1\ j+1\ k+1\ l+1\ m+1}$ & $1\otimes E_{ijklm}$ & 6 & 5 \\ \hline
     $e_{234567}$ & $1\otimes E_{123456}$ & 6 & 6 \\ \hline
     $e_{1234567}$ & $r\otimes e$ & 7 & 0 \\ \hline
     $e_{1}$ & $r\otimes E_{123456}$ & 1 & 6 \\ \hline
     $e_{234567}e_{i+1\ j+1\ k+1\ l+1\ m+1}$ & $r\otimes E_{ijklm}$ & 1 & 5 \\ \hline
     $e_{1234567}e_{i+1\ j+1\ k+1\ l+1}$ & $r\otimes E_{ijkl}$ & 3 & 4 \\ \hline
     $e_{234567}e_{i+1\ j+1\ k+1}$ & $r\otimes E_{ijk}$ & 3 & 3 \\ \hline
     $e_{1234567}e_{i+1\ j+1}$ & $r\otimes E_{ij}$ & 5 & 2 \\ \hline
     $e_{234567}e_{i+1}$ & $r\otimes E_{i}$ & 5 & 1 \\ \hline
    \end{tabular}
    \caption{Mapping of grades under the isomorphism~$\phi$ (Definition~\ref{def:phi}). Here~$i,j,k,l,m~\in~\{ 1,\ldots,6\}$.}
    \label{tab}
\end{center}
\end{table}
\begin{lemma}\label{cor:m}
Let $U\in\cl_{p,q},p+q=7$ and let $\phi$ be the isomorphism from Definition~\ref{def:phi}. Then the following identities hold
\begin{eqnarray}
        \phi(m_{0,7}(U))&=&m_{0}(\phi(U)),\quad
        \phi(m_{2,5}(U))=m_{1,2}(\phi(U)),\nonumber\\
        \phi(m_{3,4}(U))&=&m_{3,4}(\phi(U)),\quad
        \phi(m_{1,6}(U))=m_{5,6}(\phi(U)).\nonumber
    \end{eqnarray}
\end{lemma}
\begin{proof}
The Lemma follows from analyzing how the isomorphism $\phi$ from Definition~\ref{def:phi} acts on basis elements of $\cl_{p,q}$ of various grades. The identities are verified for all basis elements of $\cl_{p,q}$ listed in Table~\ref{tab}, which provides a complete description of the mapping $\phi$ on basis elements. 
\end{proof}
Note that the identities provided above do not extend to $m_{4}(\phi(U))$ and $m_{1,4,5}(\phi(U))$, which appear in the determinant formula in Theorem~\ref{n=7}. In particular, there is no subset $M\subseteq \{0, \ldots, 6\}$ for which the identities $
\phi(m_M(U)) = m_{M'}(\phi(U))
$ hold in the cases $M'=\{4\}$ and $M'=\{1, 4, 5\}$. This limitation motivates the introduction of the following generalization of conjugation operation.
\begin{definition}\label{def:bconj}
    For a subset of basis elements $B\subseteq \mathbb{B}_{p,q}$ (see~(\ref{clbasis})), we define the associated basis conjugation map $b_{B}: \cl_{p,q}\rightarrow\cl_{p,q}$
    $$
        b_{B}(U)=U-2\sum_{i\in B}[U]_i,
    $$
where $[U]_i$ is the projection onto the basis element $i$.
\end{definition}

\begin{example}
   For the element $U=e+3e_1+e_7+e_{17}+e_{1234567}\in\cl_{7,0}$ with generators $\{e_1,\ldots,e_7\}$ and a subset of basis elements $B=\{e_1,e_{1234567}\}$, we have
    \begin{eqnarray}
        &&b_{e_1,e_{1234567}}(U)=U-2[ U ]_{e_1}-2[ U ]_{e_{1234567}}=e+3e_1+e_7+e_{17}+e_{1234567}\nonumber\\
        &&-6e_1-2e_{1234567}=e-3e_1+e_7+e_{17}-e_{1234567}.\nonumber
    \end{eqnarray}
\end{example}
Note that the basis conjugation operations are a generalization of the conjugation operations and have the following properties.
\begin{lemma}
For $U,W\in\cl_{p,q}$, $\alpha, \beta \in \mathbb{R}$, $B,C\subseteq {\mathbb B}_{p,q}$ and $M\subseteq {\{1,\ldots,n\}}$, we have
\begin{eqnarray}
    &&
    \hspace{-1cm}
    b_B(\alpha U+\beta W)=\alpha b_B(U)+ \beta b_B(W),\quad
    b_B(b_B(U))=U,\\
    &&
    \hspace{-1cm}
    b_{B}(b_{C}(U))=b_{C}(b_{B}(U))=b_{B\bigtriangleup C}(U),\quad
    m_{M}(b_{B}(U))=b_{B}(m_{M}(U)),\label{propbconj}
\end{eqnarray}
where $\bigtriangleup$ denotes the symmetric difference of sets, $b_B,b_C$ are arbitrary basis conjugation operations, $m_M$ is an arbitrary conjugation operation.
\end{lemma}
This broader family of operations allows the conjugation operation to be transferred through an isomorphism $\phi$ introduced in Definition~\ref{def:phi}.
\begin{lemma}\label{lem:b}
    For $ U\in\cl_{p,q}, p+q=7$, we have
    $$
    m_s(\phi(U)) = \phi(b_{\underline{s}}(U)),\qquad s=1,2,\ldots,6,
    $$
    where 
    {
    \begin{eqnarray}
        \underline{0}&=&
        \{e,e_{1234567}\},\nonumber\\
        \underline{1}&=&
        \{e_{12},e_{13},e_{14},e_{15},e_{16},e_{17},e_{34567},e_{24567},e_{23567},e_{23467},e_{23457},e_{23456}\},\nonumber\\
        \underline{2}&=&\{e_{23},e_{24},e_{25},e_{26},e_{27},
                        e_{34},e_{35},e_{36},e_{37},
                        e_{45},e_{46},e_{47},
                        e_{56},e_{57},
                        e_{67},\nonumber\\
                        &&e_{14567},e_{13567},e_{13467},e_{13457},e_{13456},
                        e_{12567},e_{12467},e_{12457},e_{12456},\nonumber\\
                        &&
                        e_{12367},e_{12357},e_{12356},
                        e_{12347},e_{12346},
                        e_{12345}
                        \},\nonumber\\
        \underline{3}&=&\{e_{1234},e_{1235},e_{1236},e_{1237},e_{1245},e_{1246},
                        e_{1247},e_{1256},e_{1257},e_{1267},\nonumber\\
                       &&e_{1345},e_{1346},e_{1347},e_{1356},e_{1357},e_{1367},e_{1456},e_{1457},
                       e_{1467},e_{1567},\label{ubdeline}\\
                       &&e_{567},e_{467},e_{457},e_{456},e_{367},e_{357},e_{356},e_{347},e_{346},e_{345},e_{267},e_{257},\nonumber\\
                       &&e_{256},e_{247},e_{246},e_{245},e_{237},e_{236},e_{235},e_{234}
    \},\nonumber\\
        \underline{4}&=&\{e_{2345},e_{2346},e_{2356},e_{2456},e_{3456},
                        e_{2347},e_{2357},e_{2457},e_{3457},
                        e_{2367},\nonumber\\
                       &&e_{2467},e_{3467},e_{2567},e_{3567},
                        e_{4567},
                        e_{167},e_{157},e_{147},e_{137},e_{127},
                        e_{156},\nonumber\\
                       &&e_{146},e_{136},e_{126},e_{145},e_{135},e_{125},e_{134},e_{124},e_{123}
                        \},\nonumber\\
        \underline{5}&=&\{e_2,e_3,e_4,e_5,e_6,e_7,e_{134567},e_{124567},e_{123567},e_{123467},e_{123457},e_{123456}\},\nonumber\\
        \underline{6}&=&\{e_1,e_{234567}\}.\nonumber
    \end{eqnarray}     }
\end{lemma}

\begin{proof}
Table~\ref{tab} lists a all basis elements of $\cl_{p,q}, p+q=7$, parametrised by $i,j,k,l,m\in\{1,\dots,6\}$ with their images under $\phi$ (up to an overall sign) and the grades of those images.  Grouping the rows of Table~\ref{tab} according to the grade of $\phi(U)$ shows that the sets $\underline{1},\dots,\underline{6}$ are precisely the families of elements whose images have grades $1,\dots,6$. Thus, for every basis element $e_A\in\mathbb{B}_{p,q}$, we have
$$
m_s(\phi(e_A)) = 
\begin{cases}
-\phi(e_A), & \text{if } e_A\in\underline{s},\\
\phi(e_A), & \text{else}.
\end{cases}
$$
Let us have an arbitrary element $U\in\cl_{p,q}$. Let us expand the element in terms of the basis $U=\sum_{A\in{\mathbb B}_{p,q}}u_A e_A$. By linearity of $\phi$ and $m_s$, we get
\begin{align*}
m_s(\phi(U)) 
&
   = \sum_{A\in{\mathbb B}_{p,q}}u_A\,m_s(\phi(e_A)) 
   = \sum_{e_A\notin\underline{s}}u_A\,\phi(e_A)-\sum_{e_A\in\underline{s}}u_A\,\phi(e_A)\\
   &= \sum_{A\in{\mathbb B}_{p,q}}u_A\phi(b_{\underline{s}}(e_A))
   = \sum_{A\in{\mathbb B}_{p,q}}\phi(b_{\underline{s}}(u_Ae_A))
   = \phi(b_{\underline{s}}(U)).
\end{align*}
\end{proof}

\begin{lemma}\label{lem:236}
    For $U,V\in\cl_{p,q}$, where $p+q=7$, we have
    \begin{eqnarray}
        b_{\underline{2},\underline{3},\underline{6}}(UV)&=&
        b_{\underline{2},\underline{3},\underline{6}}(V)b_{\underline{2},\underline{3},\underline{6}}(U),\label{b_236p}\\
        b_{\underline{1},\underline{3},\underline{5}}(UV)&=&
        b_{\underline{1},\underline{3},\underline{5}}(U)b_{\underline{1},\underline{3},\underline{5}}(V),\label{b_135p}\\
        b_{\underline{1},\underline{2},\underline{5},\underline{6}}(UV)&=&
        b_{\underline{1},\underline{2},\underline{5},\underline{6}}(V)b_{\underline{1},\underline{2},\underline{5},\underline{6}}(U).\label{b_1256p}
    \end{eqnarray}
    where $\underline{1},\underline{2},\underline{3},\underline{5},\underline{6}$ are defined in (\ref{ubdeline}).
\end{lemma}

\begin{proof}
    Using Lemma~\ref{lem:b}, Definition~\ref{def:phi} and properties~(\ref{conjug2p}), we get  
    \begin{eqnarray}
    \phi(b_{\underline{2},\underline{3},\underline{6}}(UV))&=&
    m_{2,3,6}(\phi(UV))=
    \widetilde{\phi(UV)}=     
    \widetilde{\phi(U)\phi(V)}=
    \widetilde{\phi(V)}\widetilde{\phi(U)}\nonumber\\
    &=& 
    m_{2,3,6}(\phi({V}))m_{2,3,6}({\phi(U)})
    =
    \phi(b_{\underline{2},\underline{3},\underline{6}}({V}))\phi(b_{\underline{2},\underline{3},\underline{6}}({U}))\nonumber
    \\
    &=&
    \phi(b_{\underline{2},\underline{3},\underline{6}}({V})b_{\underline{2},\underline{3},\underline{6}}({U}))
    =
    \phi(b_{\underline{2},\underline{3},\underline{6}}({V}{U})),\label{proof1}\\
    \phi(b_{\underline{1},\underline{3},\underline{5}}(UV))
    &=&
    m_{1,3,5}(\phi(UV))=
    \widehat{\phi(UV)}=     
    \widehat{\phi(U)\phi(V)}=
    \widehat{\phi(U)}\widehat{\phi(V)}\nonumber\\
    &=& 
    m_{1,3,5}(\phi({U}))m_{1,3,5}({\phi(V)})
    =
    \phi(b_{\underline{1},\underline{3},\underline{5}}({U}))\phi(b_{\underline{1},\underline{3},\underline{5}}({V}))\nonumber\\
    &=&  
    \phi(b_{\underline{1},\underline{3},\underline{5}}({U})b_{\underline{1},\underline{3},\underline{5}}({V}))
    =
    \phi(b_{\underline{1},\underline{3},\underline{5}}({U}{V})).\label{proof2}
    \end{eqnarray}    
    Applying $\phi^{-1}$ to both sides of (\ref{proof1}) and (\ref{proof2}) yields (\ref{b_236p}) and (\ref{b_135p}), respectively. The final identity (\ref{b_1256p}) follows from (\ref{b_236p}) and (\ref{b_135p}). 
\end{proof}
This result allows us to obtain the following reformulation of Theorem~\ref{n=7}, which does not depend on the isomorphism $\phi$, meaning all computations are made directly within the algebra itself.
\begin{thm}\label{n==7}
    Let $U\in\cl_{p,q}$, where $p+q=7$. Then the determinant of the element~$U$ can be calculated by the formula
    \begin{eqnarray}
       {\rm Det}(U)&=&Uy\widehat{Uy}\in\cl^0_{p,q},\label{th3det}\\
       y&:=& 
       \frac{1}{3}b_{\underline{2},\underline{3},\underline{6}}(U)W b_{\underline{1},\underline{4},\underline{5}}(WW)\nonumber\\
       &&+\frac{2}{3}b_{\underline{2},\underline{3},\underline{6}}(U)b_{\underline{4}}(b_{\underline{4}}(W)b_{\underline{1},\underline{4},\underline{5}}(b_{\underline{4}}(W)b_{\underline{4}}(W))),\label{th3y}\\
       W&:=&U b_{\underline{2},\underline{3},\underline{6}}(U),\nonumber
    \end{eqnarray}
    where sets $\underline{1},\underline{2},\underline{3},\underline{4},\underline{5},\underline{6}$ are defined in Lemma~\ref{lem:b}.
\end{thm}
\begin{proof}

Let $U\in\cl_{p,q}$ with $p+q=7$, and let $\phi$ be the isomorphism given in (\ref{phi}). For any $X\in\cl_{p',q'}$ with $p'+q'=6$, denote by $\overline{X}$ the result of applying the inverse isomorphism $\phi^{-1}$ to $X$, meaning $\overline{X}:=\phi^{-1}(X)$. By Theorem~\ref{n=7}, the determinant of $U$ can be written as
\begin{gather}
    {\rm Det}(U) = \overline{B} \overline{y}\widehat{\overline{B}\overline{y}},\quad 
    V=B\widetilde{B},\quad
    B=\phi(U),\label{n==7:insert1}\\
    \overline{y}=\overline{\widetilde{B}}\left(
    \frac{1}{3}\overline{V}\overline{m_{1,4,5}(VV)}+
    \frac{2}{3}\overline{m_4(m_4(V)m_{1,4,5}(m_4(V)m_4(V)))}
    \right)\label{n==7:insert2}
\end{gather}
We now simplify each part of the formula using Lemma~\ref{lem:b}.
\begin{eqnarray}
    \overline{\widetilde{B}}&=&
    \overline{m_{2,3,6}(\phi(U))}=
    \overline{\phi\left(b_{\underline{2},\underline{3},\underline{6}}(U)\right)}=b_{\underline{2},\underline{3},\underline{6}}(U)\label{n==7:t1}\\
    \overline{V}&=&
    \overline{B}\overline{\widetilde{B}}=
    \overline{\phi\left(U\right)}\,\overline{\phi\left(b_{\underline{2},\underline{3},\underline{6}}(U)\right)}=
    Ub_{\underline{2},\underline{3},\underline{6}}(U)=W\label{n==7:t2}
\end{eqnarray}
\begin{eqnarray}
    \overline{m_{1,4,5}(VV)}&=&
    \overline{m_{1,4,5}\left(\phi(\overline{VV})\right)}
    =
    \overline{\phi\left(b_{\underline{1},\underline{4},\underline{5}}(\overline{V\,V})\right)}\nonumber\\&=&
    b_{\underline{1},\underline{4},\underline{5}}(\overline{V}\,\overline{V})=
    b_{\underline{1},\underline{4},\underline{5}}(WW)\label{n==7:t3}
\end{eqnarray}

\begin{eqnarray}
    \overline{m_4(m_4(V)m_{1,4,5}(m_4(V)m_4(V)))}&=&
    \overline{m_4\left(\phi\left(\overline{m_4(V)}\overline{m_{1,4,5}(m_4(V)m_4(V)}\right)\right)}\nonumber\\
    &=&
    \overline{\phi\left(b_{\underline{4}}\left(\overline{m_4(V)}\,\overline{m_{1,4,5}(m_4(V)m_4(V))}\right)\right)}\nonumber\\
    &=&
    b_{\underline{4}}\left(\overline{m_4(V)}\,\overline{m_{1,4,5}(m_4(V)m_4(V))}\right)\label{n==7:t4}\\
    &=&
    b_{\underline{4}}\left(\overline{m_4\left(\phi\left(\overline{V}\right)\right)}\,
    \overline{m_{1,4,5}\left(\phi\left(\overline{m_4(V)m_4(V)}\right)\right)}\right)\nonumber\\
    &=&
    b_{\underline{4}}\left(b_{\underline{4}}\left(\overline{V}\right)\,
    b_{\underline{1},\underline{4},\underline{5}}\left(\overline{m_4(V)}\,\overline{m_4(V)}\right)\right)\nonumber\\
    &=&
    b_{\underline{4}}\left(b_{\underline{4}}\left(W\right)\,
    b_{\underline{1},\underline{4},\underline{5}}\left(b_{\underline{4}}\left(W\right)b_{\underline{4}}\left(W\right)\right)\right)\nonumber
\end{eqnarray}
Finally, inserting (\ref{n==7:t1}), (\ref{n==7:t2}), (\ref{n==7:t3}) and (\ref{n==7:t4}) into (\ref{n==7:insert1}) and( \ref{n==7:insert2}) yields the result.
\end{proof}
\begin{corollary}\label{cor:n==7}
    If ${\rm Det}(U)\neq 0$, then the inverse of $U\in\cl_{p,q}, p+q=7$ is given by
    $$
       U^{-1} = \frac{y\widehat{Uy}}{{\rm Det}(U)},
    $$
    where ${\rm Det}$ and $y$ are defined in (\ref{th3det}) and (\ref{th3y}), respectively. 
\end{corollary}

\begin{example}
    We illustrate the application of Corollary~\ref{cor:n==7} by computing the inverse\footnote{
One of the reviewers noted that for this specific \(U\), the minimal polynomial  has degree~3, so the determinant and inverse could be  found in only 3 multiplicative steps  (see \cite{Prodanov}). We follow the general corollary instead, to illustrate the systematic procedure.} of the multivector $U=e + e_{35} + e_{1234567}$. First, let us compute the auxiliary element
    \begin{align*}
        W &= U  b_{\underline{2},\underline{3},\underline{6}}(U) 
        = (e + e_{35} + e_{1234567})(e - e_{35} + e_{1234567})= e + 2e_{1234567}.
    \end{align*}
     Next, let us calculate intermediate quantities
    \begin{align*}
        WW&=  (e + 2e_{1234567})(e + 2e_{1234567}) = -3e + 4e_{1234567}, \\
        W  b_{\underline{1},\underline{4},\underline{5}}(WW) 
        &= (e + 2e_{1234567})(-3e + 4e_{1234567}) 
        = -11e - 2e_{1234567}. 
    \end{align*}
     Note that 
    \begin{align*}
        W &= b_{\underline{4}}(W), \quad
        W  b_{\underline{1},\underline{4},\underline{5}}(WW) 
        = b_{\underline{4}}(b_{\underline{4}}(W)b_{\underline{1},\underline{4},\underline{5}}(b_{\underline{4}}(W)b_{\underline{4}}(W)),\\
        y &= b_{\underline{2},\underline{3},\underline{6}}(U) \left( 
        \frac{1}{3} W  b_{\underline{1},\underline{4},\underline{5}}(WW) 
        + \frac{2}{3} b_{\underline{4}}\left( b_{\underline{4}}(W)  
        b_{\underline{1},\underline{4},\underline{5}}\left( b_{\underline{4}}(W)b_{\underline{4}}(W) \right) \right) \right) \\
        &= (e - e_{35} + e_{1234567})(-11e - 2e_{1234567})\\
        &= -9e + 11e_{35} + 2e_{12467} - 13e_{1234567}.
    \end{align*}
    Now, let us compute
    \begin{align*}
        Uy &= (e + e_{35} + e_{1234567})(-9e + 11e_{35} + 2e_{12467} - 13e_{1234567}) \\
        &= -7e - 24e_{1234567}, \\
        y  \widehat{Uy} &= (-9e + 11e_{35} + 2e_{12467} - 13e_{1234567})(-7e + 24e_{1234567}) \\
        &= 375e - 125e_{35} + 250e_{12467} - 125e_{1234567}.
    \end{align*}
    Thus, the determinant of $U$ is
    $$
        \operatorname{Det}(U) = Uy  \widehat{Uy} =(-7e - 24e_{1234567})(-7e + 24e_{1234567}) =625e.
    $$
    Finally, the inverse of $U$ is
    \begin{eqnarray*}
        U^{-1} 
        &=& \frac{y\widehat{Uy}}{{\rm Det}(U)}=\frac{3}{5}e - \frac{1}{5}e_{35} + \frac{2}{5}e_{12467} - \frac{1}{5}e_{1234567}.
    \end{eqnarray*}
\end{example}

As noted above, the previous example could be computed using a specialized shortcut for both the determinant and the inverse. To better motivate the full generality of Corollary~\ref{cor:n==7}, we now consider a multivector for which such simplification does not apply. 

\begin{example}
    Let $U = e + e_{7} + e_{13} + e_{14} + e_{15} + e_{26} + e_{34} \in \cl_{7,0}$. Applying Corollary~\ref{cor:n==7}, we first compute the auxiliary element
    \begin{eqnarray*}
        W&=&Ub_{\underline{2},\underline{3},\underline{6}}(U)=e + 2e_{7} + 4e_{13} + 2e_{15} + 2e_{137} + 2e_{147} + 2e_{157} - 2e_{2346}.
    \end{eqnarray*}
    Next, let us calculate intermediate quantities
    \begin{eqnarray*}
        W^2&=&-23e - 20e_{7} + 16e_{13} + 8e_{14} + 12e_{15} + 20e_{137} + 4e_{147} + 12e_{157}\\
        &&- 4e_{2346} - 8e_{23467} + 8e_{123456} + 8e_{1234567},\\
        b_{\underline{4}}(W)^2&=&-23e + 28e_{7} - 8e_{14} - 4e_{15} + 12e_{137} - 4e_{147} + 4e_{157} \\
        &&+ 4e_{2346} + 8e_{23467} - 8e_{123456} + 8e_{1234567},\\
        T_1&:=&Wb_{\underline{1},\underline{4},\underline{5}}(W^2)=169 + 134e_{7} - 108e_{13} + 24e_{14} - 58e_{15} \\
        &&- 18e_{137} - 66e_{147} - 26e_{157} + 66e_{2346} - 16e_{2356} - 16e_{2456} \\
        &&- 24e_{23467} + 16e_{23567} + 16e_{24567} - 40e_{123456} - 56e_{1234567},\\
        T_2&:=&b_{\underline{4}}(W)b_{\underline{1},\underline{4},\underline{5}}\left((b_{\underline{4}}(W)^2\right)=-119 - 10e_{7} - 60e_{13} + 72e_{14} - 10e_{15} \\
        &&- 78e_{137} + 66e_{147} - 22e_{157} - 66e_{2346} + 16e_{2356} + 16e_{2456} \\
        &&- 72e_{23467} + 16e_{23567} + 16e_{24567} + 8e_{123456} + 40e_{1234567}.
    \end{eqnarray*}
    Now, let us compute
    \begin{eqnarray*}
        y_0&:=&\frac{1}{3}(T_1 + 2b_4(T_2)) =-23 + 38e_{7} - 76e_{13} + 56e_{14} - 26e_{15} + 46e_{137} \\
        &&- 66e_{147} + 6e_{157} + 66e_{2346} - 16e_{2356} - 16e_{2456} - 56e_{23467} \\
        &&+ 16e_{23567} + 16e_{24567} - 8e_{123456} + 8e_{1234567},\\
        y&=&b_{\underline{2},\underline{3},\underline{6}}(U)y_0=61 + 29e_{7} - 109e_{13} - 109e_{14} - 43e_{15} + 89e_{26} - 43e_{34} \\
        &&- 66e_{35} + 66e_{45} + 74e_{137} + 74e_{147} + 18e_{157} - 94e_{267} + 18e_{347} \\
        &&+ 56e_{357} - 56e_{457} - 26e_{1236} - 26e_{1246} - 2e_{1256} + 34e_{1345} + 2e_{2346} \\
        &&+ 24e_{2356} - 24e_{2456} + 6e_{12367} + 6e_{12467} - 18e_{12567} - 14e_{13457} \\
        &&+ 18e_{23467} - 24e_{23567} + 24e_{24567} - 66e_{123456} + 56e_{1234567},\\
        Uy&=&305 - 80e_{1234567}.
    \end{eqnarray*}
    Thus, the determinant of $U$ is
    \begin{eqnarray*}
        Det(U)=Uy\widehat{Uy}=99425.\\
    \end{eqnarray*}
    Finally, the inverse of $U$ is
    \begin{eqnarray*}
        U^{-1}&=& \frac{1}{{\rm Det}(U)}y\widehat{Uy}=\frac{1}{99425}\left(14125(e + e_{7}) - 35165(e_{13} + e_{14})\right.\\
        &&- 14555(e_{15} + e_{34}) + 28265e_{26}  - 20610(e_{35} - e_{45}) \\
        &&+ 24490(e_{137} + e_{147}) + 5330(e_{157} + e_{347})- 31390e_{267}\\
        &&+ 19160(e_{357} - e_{457}) - 3450(e_{1236} + e_{12367}+ e_{1246}+ e_{12467}) \\
        &&- 2050(e_{1256}+e_{12567}- e_{2346}-e_{23467}) \\
        &&+ 2850(e_{1345} + e_{13457}) + 1400(e_{2356} - e_{2456}      + e_{23567} - e_{24567}) \\
        &&- \left.17810e_{123456} + 21960e_{1234567}\right).
    \end{eqnarray*}
\end{example}

\section{Optimized Explicit Formula of Determinant for $n=7$}
The explicit determinant formula for n=7, derived in the previous section,  is computationally intensive due to repeated application of basis conjugation operations. In this section, we present an optimized version of the formula. The optimization is particularly useful for practical computations, where classical operations of conjugation might be easier to implement.

\begin{lemma}\label{lem:uvu}
    For $U,V\in\cl_{p,q}$ with $p+q\in\{6,7\}$, we have 
    \begin{equation}
        m_{1,2,3,4,5,6}(UV)U=Um_{1,2,3,4,5,6}(VU).\label{uvu}
    \end{equation}
\end{lemma}
\begin{proof}
    First, let us consider the case $n=6$. From (\ref{comm}), we have
    \begin{eqnarray*}      
        \frac{1}{2}(UV+m_{1,2,3,4,5,6}(UV))
        =
        \langle UV\rangle_0
        =
        \langle VU\rangle_0
        =
        \frac{1}{2}(VU+m_{1,2,3,4,5,6}(VU)).
    \end{eqnarray*}
    The expressions on the left side and on the right side are scalars. Multiplying the left side by $2U$ on the right and the right side by $2U$ on the left, we get (\ref{uvu}) for $p+q=6$.

    Now, let us consider the case $n=7$. Using Lemma~\ref{lem:b}, Definition~\ref{def:phi} and just proven property (\ref{uvu}) for the case $n=6$, we get  
    \begin{eqnarray}
        &&\phi(Um_{1,2,3,4,5,6}(VU))
        =
        \phi(U)m_{1,2,3,4,5,6}(\phi(V)\phi(U)))\nonumber\\
        &&=
        \phi(b_{\underline{1},\underline{2},\underline{3},\underline{4},\underline{5},\underline{6}}(UV)U)
        =
        \phi(m_{1,2,3,4,5,6}(UV)U).\label{proof3}
    \end{eqnarray} 
    Finally, applying $\phi^{-1}$ to both sides of (\ref{proof3}) results in the (\ref{uvu}) for $p+q=7$.
\end{proof}
Note that similar results to Lemma~\ref{lem:uvu} can be found in \cite{Shirokov}.

\begin{lemma}\label{lem:mb}
    Let $U\in\cl_{p,q}$, $p+q=7$, $W:=U b_{\underline{2},\underline{3},\underline{6}}(U)$ and  $H:=m_{3,4}(W)$. Then the following identities hold 
    \begin{eqnarray}       
        b_{\underline{4}}(W)&=&m_{3,4}(W),\label{lemb1}\\
        b_{\underline{1},\underline{4},\underline{5}}(WW)&=&m_{1,2,3,4,5,6}(WW),\label{lemb2}\\
        b_{\underline{1},\underline{4},\underline{5}}(HH)&=&m_{1,2,3,4,5,6}(HH),\label{lemb3}\\
        b_{\underline{4}}(Hm_{1,2,3,4,5,6}(HH))
        &=&
        m_{3,4}(Hm_{1,2,3,4,5,6}(HH)),\label{lemb4}
    \end{eqnarray}
    where sets $\underline{1},\underline{2},\underline{3},\underline{4},\underline{5},\underline{6}$ are defined in (\ref{ubdeline}).
\end{lemma}
\begin{proof}
    By Lemma~\ref{lem:236}, we have the following identities
    \begin{eqnarray*}
        b_{\underline{2},\underline{3},\underline{6}}(W)=W,\quad b_{\underline{2},\underline{3},\underline{6}}(WW)=WW,\quad
    b_{\underline{2},\underline{3},\underline{6}}(HH)=HH.
    \end{eqnarray*}
    If we combine them with the commutative properties of the basis conjugation operations~(\ref{propbconj}), Lemma~\ref{cor:m} and Lemma~\ref{lem:b}, we obtain identities~(\ref{lemb1}), (\ref{lemb2}) and~(\ref{lemb3}) in the following way
    \begin{eqnarray*}       
        b_{\underline{4}}(W)&=&
        b_{\underline{4}}(b_{\underline{3}}(W))=
        b_{\underline{3},\underline{4}}(W)=m_{3,4}(W),\\
        b_{\underline{1},\underline{4},\underline{5}}(WW)&=&
        b_{\underline{1},\underline{4},\underline{5}}(b_{\underline{2},\underline{3},\underline{6}}(WW))
        =
        b_{\underline{1},\underline{2},\underline{3},\underline{4},\underline{5},\underline{6}}(WW)
        =
        m_{1,2,3,4,5,6}(WW),\\
        b_{\underline{1},\underline{4},\underline{5}}(HH)&=&
        b_{\underline{1},\underline{4},\underline{5}}(b_{\underline{2},\underline{3},\underline{6}}((HH))
        =
        b_{\underline{1},\underline{2},\underline{3},\underline{4},\underline{5},\underline{6}}(HH)
        =
        m_{1,2,3,4,5,6}(HH).
    \end{eqnarray*}
    
    Now, using Lemma~\ref{lem:b}, Lemma~\ref{lem:uvu} and Definition~\ref{def:phi}, we get  
    \begin{eqnarray}
    \phi(Hm_{1,2,3,4,5,6}(HH))
    &=&
    \phi(Hb_{\underline{1},\underline{2},\underline{3},\underline{4},\underline{5},\underline{6}}(HH))\nonumber\\
    &=&
    \phi(H)m_{1,2,3,4,5,6}(\phi(H)\phi(H)))\nonumber\\
    &=&
    m_{1,2,3,4,5,6}(\phi(H)\phi(H))\phi(H)\label{proof4}\\
    &=&
    \phi(b_{\underline{1},\underline{2},\underline{3},\underline{4},\underline{5},\underline{6}}(HH)H)\nonumber\\
    &=&
    \phi(m_{1,2,3,4,5,6}(HH)H).\nonumber
    \end{eqnarray} 
        Applying $\phi^{-1}$ to both sides of (\ref{proof4}) results in the following identity 
    \begin{equation}
        Hm_{1,2,3,4,5,6}(HH)=m_{1,2,3,4,5,6}(HH)H.\label{vvv}
    \end{equation}
    Next, using commutative properties of the basis conjugation operations~(\ref{propbconj}), Lemma~\ref{lem:236}, Lemma~\ref{lem:uvu} and~(\ref{vvv}), we obtain
    \begin{eqnarray}
        &&b_{\underline{2},\underline{3},\underline{6}}\left(b_{\underline{4}}(Hm_{1,2,3,4,5,6}(HH))\right)
        =
        b_{\underline{4}}\left(b_{\underline{2},\underline{3},\underline{6}}(Hm_{1,2,3,4,5,6}(HH))\right)\nonumber\\
        &&=
        b_{\underline{4}}\left(b_{\underline{2},\underline{3},\underline{6}}(m_{1,2,3,4,5,6}(HH))b_{\underline{2},\underline{3},\underline{6}}(H)\right)
        \nonumber\\
        &&=
        b_{\underline{4}}\left(m_{1,2,3,4,5,6}(b_{\underline{2},\underline{3},\underline{6}}(HH))b_{\underline{2},\underline{3},\underline{6}}(H)\right)\nonumber\\
        &&=
        b_{\underline{4}}(m_{1,2,3,4,5,6}(HH)H)
        =
        b_{\underline{4}}(Hm_{1,2,3,4,5,6}(HH)).\label{proof5}
    \end{eqnarray}
    In~(\ref{proof5}), we see that the expression~$b_{\underline{4}}(Hm_{1,2,3,4,5,6}(HH))$ does not change under $b_{\underline{2}, \underline{3}, \underline{6}}$, which in particular means that it does not change under $b_{\underline{3}}$. This leads to identity~(\ref{lemb4}) in the following manner
    \begin{eqnarray}
        b_{\underline{4}}\left(Hm_{1,2,3,4,5,6}(HH)\right)=
        b_{\underline{3}}\left(b_{\underline{4}}(Hm_{1,2,3,4,5,6}(HH))\right)=     
        m_{3,4}(Hm_{1,2,3,4,5,6}(HH)).\nonumber
    \end{eqnarray}
\end{proof}

\begin{thm}\label{n===7}
    Let $U\in\cl_{p,q}$, where $p+q=7$. Then the determinant of the element $U$ can be calculated by the formula
    \begin{eqnarray}
       {\rm Det}(U)&=&Uy\widehat{Uy}\in\cl^0_{p,q},\label{th4det}
    \end{eqnarray}
    where 
    \begin{eqnarray}
       y&:=& b_{\underline{2},\underline{3},\underline{6}}(U)\left(
       \frac{1}{3}W m_{1,2,3,4,5,6}(WW)\right.\label{th4y}\\
       &+&\left.\frac{2}{3}m_{3,4}(m_{3,4}(W)m_{1,2,3,4,5,6}(m_{3,4}(W)m_{3,4}(W)))
       \right),\nonumber\\
       W&:=&U b_{\underline{2},\underline{3},\underline{6}}(U).\nonumber
    \end{eqnarray}
    Here sets $\underline{2},\underline{3},\underline{6}$ are defined in (\ref{ubdeline}).
\end{thm}
\begin{proof}
    The result follows from Theorem \ref{n==7} and Lemma \ref{lem:mb}.
\end{proof}
\begin{corollary}
    If ${\rm Det}(U)\neq 0$, then the inverse of $U\in\cl_{p,q}, p+q=7$ is given by
    $$
       U^{-1} = \frac{y\widehat{Uy}}{{\rm Det}(U)},
    $$
    where ${\rm Det}$ and $y$ are defined in (\ref{th4det}) and (\ref{th4y}), respectively. 
\end{corollary}

Note that the optimized version of the formulas yields better results when the conjugation operation itself is better optimized. This improvement arises because conjugation operations are often highly optimized in many programming libraries, as we demonstrate in Section~\ref{sec:num}. Furthermore, the use of classical conjugation might be theoretically preferable, as it minimizes dependence on specific basis elements.
\section{Numerical Implementation and Performance of the Formulas for $n=7$}\label{sec:num}

In this section, we provide a numerical implementation of the formulas and evaluate their performance. Specifically, we implement the formulas for the determinant and inverse in the geometric algebras $\cl_{p,q}$ with $p + q = 7$, as derived in earlier sections using formal algebraic methods. For completeness, we also implement a recursive method based on the Faddeev–LeVerrier algorithm. The code implementation is available at the repository \url{github.com/kamranuz/clifford_7d}.

Hitzer's and Sangwine's isomorphism-dependent formulas from Theorem~\ref{n=7} are implemented in the functions \texttt{cl7\_determinant\_hitzersangwine} and \texttt{cl7\_inverse\_hitzersangwine}. Recursive method  based on Faddeev-LeVerrier algorithm \cite{Shirokov} is implemented in the functions \texttt{cl7\_inverse\_fvs} and \texttt{cl7\_determinant\_fvs}. Since in our setup a random multivector is multivector of span dimension 7 with probability one, recursive method with a variable number of steps \cite{Prodanov} is equivalent to \texttt{cl7\_determinant\_fvs} and \texttt{cl7\_inverse\_fvs}. Our explicit formulas from Theorem~\ref{n==7} are implemented in the functions \texttt{cl7\_determinant} and \texttt{cl7\_inverse}, while optimized versions from Theorem~\ref{n===7} are provided in functions \texttt{cl7\_determinant\_opt} and \texttt{cl7\_inverse\_opt}. 

The main difference between optimized and non-optimized versions lies in the choice of conjugation operations. Consequently, performance depends heavily on how the Clifford conjugation is implemented. To illustrate this, we compare two distinct approaches: \textit{the projection method}, which subtracts specific components from the original multivector, and \textit{the signed array method}, which applies conjugation through element-wise multiplication of basis elements with a predefined array of signs. The functions \texttt{cl7\_determinant}, \texttt{cl7\_determinant\_opt}, \texttt{cl7\_inverse}, and \texttt{cl7\_inverse\_opt} are implemented using first approach, while  \texttt{cl7\_ar\_inverse\_opt}, \texttt{cl7\_ar\_determinant\_opt}, \texttt{cl7\_ar\_determinant} and \texttt{cl7\_ar\_inverse} use the second. 

All the experiments were conducted on Google Compute Engine virtual instances via Google Colab with 12.7\,GB of RAM and a dual‑core \texttt{Intel(R) Xeon(R) CPU @ 2.20GHz)}, with hyper‑threading disabled. The implementation environment consisted of Python 3.12.13, the \texttt{clifford} Python library 1.5.1~\cite{Python} for geometric algebra construction and multivector operations. Additionally, standard and scientific Python libraries such as \texttt{numpy} 2.0.2, \texttt{timeit}, \texttt{gc}, and \texttt{tracemalloc} were employed for computation and performance profiling.

We first validated the correctness of the implementation using unit tests. These tests confirmed that our determinant and inverse functions yield results consistent with algebraic expectations. However, we discovered that due to limits of floating-point precision, \texttt{cl7\_determinant\_hitzersangwine} and  \texttt{cl7\_inverse\_hitzersangwine}  do not produce correct results. Particularly, \texttt{cl7\_determinant\_hitzersangwine} returns a non-zero scalar accompanied by a non-scalar part close to zero, while \texttt{cl7\_inverse\_hitzersangwine} returns an element that does not satisfy the defining inverse property. While the first function we rectified by clipping coefficients close to zero, the second we could not correct.

We then evaluated performance. A test set of 100\,000 multivectors, each having 128 normally distributed coefficients, with a fixed NumPy random seed to ensure reproducibility. Each function was applied to every multivector in this set, and the resulting performance metrics are summarized in Table~\ref{tab:per} and Figure~\ref{fig:per}. For each candidate function, the evaluation proceeded as follows: garbage collection was disabled to avoid unpredictable timing interruptions; each multivector in the fixed test set was processed sequentially; and the execution time per call was measured using \texttt{time.perf\_counter}. After all runs completed, peak memory usage and memory allocation traces were recorded, and garbage collection was turned on.

Function \texttt{cl7\_inverse\_hitzersangwine} was excluded from the subsequent performance evaluation as it failed the unit tests. The performance results illustrate a clear trade-off between computational efficiency and numerical robustness. According to Hitzer’s and Sangwine's formula, the core calculations occur in a subalgebra of reduced dimension. Computations within this smaller algebra are faster due to the lower number of coefficients required (64 versus 128 in the full algebra), explaining the superior speed of \texttt{cl7\_determinant\_hitzersangwine}. However, this theoretical speed advantage is negated in practice by numerical instability, which manifests as non-scalar artifacts in the determinant and a complete failure in the inverse computation.

We explain this instability by compounded floating-point errors, which arise primarily from Python's standard implementation of complex numbers required for evaluating the isomorphism~$\phi$. In contrast, the formulas we have derived do not require an explicit implementation of the isomorphism~$\phi$, nor do they rely complex or hyperbolic numbers. This not only simplifies implementation but also enhances numerical stability by avoiding error-prone intermediate mappings.

The performance results demonstrate that implementations based on the signed array method are among the fastest, with mean execution times ranging from 1.76\,ms to 2.27\,ms and low standard deviations. The near-identical performance of the optimised and non-optimised variants in this group confirms that the signed-array approach is already highly efficient, leaving essentially no room for further optimisation. In contrast, the projection method benefits substantially from the optimised conjugation strategy: the optimised versions (\texttt{cl7\_determinant\_opt} and \texttt{cl7\_inverse\_opt}) achieve a roughly twofold speedup over their non-optimised counterparts (\texttt{cl7\_determinant} and \texttt{cl7\_inverse}), reducing mean times from approximately 14\,ms to 6.5\,ms. The recursive method based on the Faddeev–LeVerrier algorithm (\texttt{cl7\_determinant\_fvs} and \texttt{cl7\_inverse\_fvs}) are faster than implementations with the projection method but slower than implementation with the signed-array method and isomorphism-based variants.

\begin{table}[h!]
\centering
\label{tab:per}
\begin{tabular}{|l|c|c|}
\hline
\textbf{Function} & \textbf{Mean Time} & \textbf{Std Dev }\\ 
& \textbf{(ms)} & \textbf{(ms)} \\\hline
\texttt{cl7\_determinant\_hitzersangwine} & 1.534 & 0.478 \\\hline
\texttt{cl7\_ar\_determinant\_opt} & 1.760 & 0.458 \\\hline
\texttt{cl7\_ar\_determinant} & 1.765 & 0.461 \\\hline
\texttt{cl7\_ar\_inverse} & 2.264 & 0.575 \\\hline
\texttt{cl7\_ar\_inverse\_opt} & 2.273 & 0.585 \\\hline
\texttt{cl7\_determinant\_fvs} & 4.447 & 1.069 \\\hline
\texttt{cl7\_inverse\_fvs} & 5.554 & 1.338 \\\hline
\texttt{cl7\_determinant\_opt} & 6.276 & 1.455 \\\hline
\texttt{cl7\_inverse\_opt} & 6.736 & 1.553 \\\hline
\texttt{cl7\_determinant} & 13.735 & 3.107 \\\hline
\texttt{cl7\_inverse} & 14.118 & 3.185 \\\hline
\end{tabular}
\caption{Performance Metrics (Averaged Over 100\,000 Runs)}
\end{table}

\begin{figure}[h!]
    \centering
    \begin{tikzpicture}
        \begin{axis}[
            title={},
            width=0.95\textwidth,
            height=9cm,
            bar width=13pt,
            symbolic x coords={
                cl7\_determinant\_hitzersangwine,
                cl7\_ar\_determinant\_opt,
                cl7\_ar\_determinant,
                cl7\_ar\_inverse,
                cl7\_ar\_inverse\_opt,
                cl7\_determinant\_fvs,
                cl7\_inverse\_fvs,
                cl7\_determinant\_opt,
                cl7\_inverse\_opt,
                cl7\_determinant,
                cl7\_inverse,
            },
            xtick=data,
            nodes near coords,
            nodes near coords align={vertical},
            ylabel={Mean Time (ms)},
            xlabel={},
            xticklabel style={rotate=45, anchor=east, font=\tiny},
            ymajorgrids,
            grid style={dashed, gray!30},
            error bars/y dir=both,
            error bars/y explicit,
            ymax=18,
            ymin=0
        ]
            \addplot[ybar, fill=gray!0, draw=black, error bars/.cd, y explicit, error bar style={black, thick}] coordinates {
                (cl7\_determinant\_hitzersangwine, 1.534)
                (cl7\_determinant\_fvs,           4.447) 
                (cl7\_inverse\_fvs,               5.554) 
                (cl7\_determinant,               13.735) 
                (cl7\_determinant\_opt,           6.276) 
                (cl7\_inverse,                   14.118) 
                (cl7\_inverse\_opt,               6.736) 
                (cl7\_ar\_determinant,            1.765) 
                (cl7\_ar\_determinant\_opt,       1.760) 
                (cl7\_ar\_inverse,                2.264) 
                (cl7\_ar\_inverse\_opt,           2.273) 
            };
            \addplot[ybar, fill=black!60, draw=black, error bars/.cd, y explicit, error bar style={black, thick}] coordinates {
                (cl7\_determinant\_hitzersangwine, 1.534) +- (0, 0.478)
            };
            \addplot[ybar, fill=white!100, draw=black, error bars/.cd, y explicit, error bar style={black, thick}] coordinates {
                (cl7\_determinant\_fvs,           4.447) +- (0, 1.069)
                (cl7\_inverse\_fvs,               5.554) +- (0, 1.338)
            };
            \addplot[ybar, fill=gray!70, draw=black, error bars/.cd, y explicit, error bar style={black, thick}] coordinates {
                (cl7\_determinant,               13.735) +- (0, 3.107)
                (cl7\_inverse,                   14.118) +- (0, 3.185)
            };
            \addplot[ybar, fill=gray!40, draw=black, error bars/.cd, y explicit, error bar style={black, thick}] coordinates {
                (cl7\_determinant\_opt,           6.276) +- (0, 1.455)
                (cl7\_inverse\_opt,               6.736) +- (0, 1.553)
            };
            
            \addplot[ybar, fill=green!40, draw=black, error bars/.cd, y explicit, error bar style={black, thick}] coordinates {
                (cl7\_ar\_determinant\_opt,       1.760) +- (0, 0.458)
                (cl7\_ar\_inverse\_opt,           2.273) +- (0, 0.585)
            };
            
            \addplot[ybar, fill=green!70, draw=black, error bars/.cd, y explicit, error bar style={black, thick}] coordinates {
                (cl7\_ar\_determinant,            1.765) +- (0, 0.461)
                (cl7\_ar\_inverse,                2.264) +- (0, 0.575)
            };
        \end{axis}
    \end{tikzpicture}
    \caption{Performance Metrics (Averaged Over 100\,000 Runs)}
    \label{fig:per}
\end{figure}
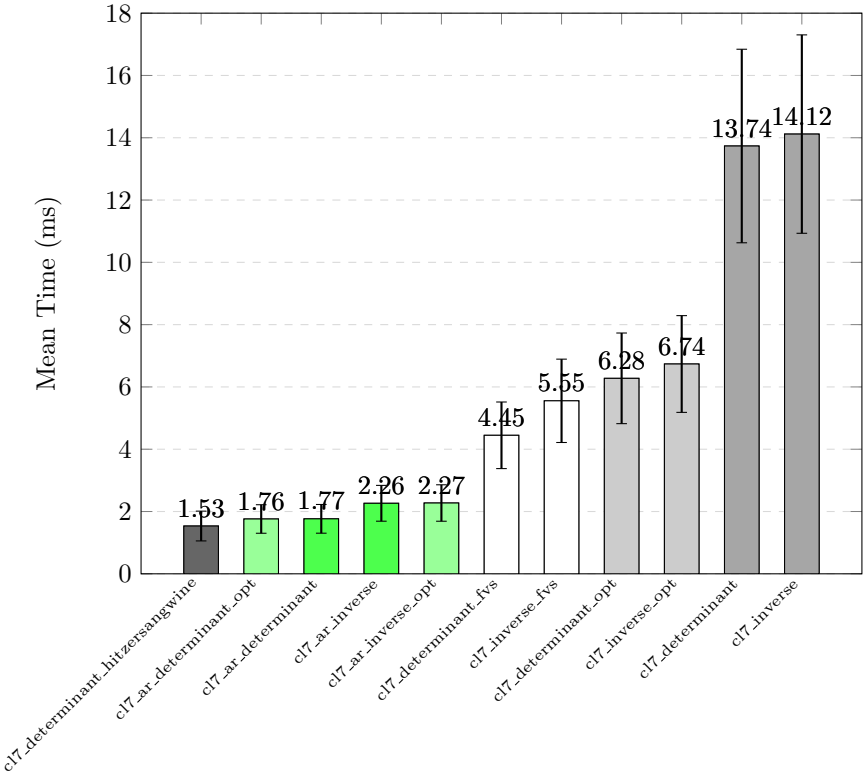

\section{Explicit Formula for the case $n=2k+1$ from $n=2k$}

In this section, we present a method to derive explicit formulas for the determinant and inverse in the for Clifford algebras over $2k + 1$-dimensional vector spaces, using known formulas for $2k$-dimensional vector spaces case\footnote{
One of the reviewers noted that there is no analogous step from odd dimension to the next even dimension for the following reason. The tensor products $\mathbb{C}\otimes\cl_{p,q}$ and $\mathbb{S}\otimes\cl_{p,q}$ are equivalent to $\cl_{0,1} \otimes \cl_{p,q}$ and $\cl_{1,0} \otimes \cl_{p,q}$ respectively. However, when two or more one-dimensional Clifford algebras appear, such a tensor product no longer yields a new Clifford algebra. Instead, it describes an extension \cite{Marchuk}.}. First, let us restate determinant and inverse formulas for $n=2k+1$ obtained by Hitzer and Sangwine using a technique based on algebra isomorphisms~\cite{Hitzer2}.
\begin{definition}[\cite{Hitzer2}]
\label{def:Phi}
    Let $\{e_1,\ldots,e_n\}$ be the generators of $\cl_{p,q}$ with $p+q=n=2k+1$ and $\{E_1,\ldots,E_{n-1}\}$ be the generators of $\cl_{p',q'}$ with $p'+q'=n-1=2k$. Let us define the isomorphism 
    \begin{eqnarray}    
        &&\Phi:\cl_{p,q}\xrightarrow{\sim}
        \begin{cases}
            \mathbb{C}\otimes\cl_{q,p-1},\quad \text{if $k+q$ is odd},\nonumber\\
            \mathbb{S}\otimes\cl_{q,p-1},\quad \text{else},
        \end{cases}   \nonumber\\
        &&\Phi(e_{1~m+1}):=1\otimes E_m,\quad m=1,\ldots,n-1,\nonumber\\
        &&\Phi(e_{12\ldots n}):=r\otimes E,\quad r=\begin{cases}
            i\in\mathbb{C},\quad \text{if $k+q$ is odd},\\
            j\in\mathbb{S},\quad \text{else},
        \end{cases}\nonumber
    \end{eqnarray}
    where $E$ is the identity element of the algebra $\cl_{p',q'}$.
    \end{definition}
\begin{thm}[\cite{Hitzer2}]
    Let us have $U\in\cl_{p',q'},p'+q'=2k$ and
    \begin{eqnarray*}
    {\rm Det}_{2k}(U)=Uf(U)\in\cl_{p',q'}^0,
    \end{eqnarray*}
    where $f(U)$ is a sum of products of conjugations. Then, for $V\in\cl_{p,q}, p+q~=~2k+1$, we have
    \begin{eqnarray}
    {\rm Det}_{2k+1}(V)=\Phi^{-1}({\rm Det}_{2k}(B))\widehat{\Phi^{-1}}({\rm Det}_{2k}(B)),\quad B:=\Phi(V).\label{thm:gen1}
    \end{eqnarray}
\end{thm}
\begin{corollary}[\cite{Hitzer2}]
    If ${\rm Det}(V)\neq 0$, then the inverse of $V\in\cl_{p,q}, p+q=2k+1$ is given by
    $$
       V^{-1} = \frac{\Phi^{-1}(f(B)))\widehat{\Phi^{-1}}({\rm Det}_{2k}(B))}{{\rm Det}_{2k+1}(V)},
    $$
    where ${\rm Det}$ and $B$ are defined in (\ref{thm:gen1}).
\end{corollary}

Now, we present a method for obtaining explicit and isomorphism-invariant formulas for the determinant and inverse in the following theorem.
\begin{thm}
    Let us have $U\in\cl_{p',q'},p'+q'=2k$ and
    \begin{eqnarray*}
    {\rm Det}_{2k}(U)=Uf(U)\in\cl_{p',q'}^0,
    \end{eqnarray*}
    where $f(U)$ is a sum of products of conjugations. Then, for $V\in\cl_{p,q}, p+q=2k+1$, we have
    \begin{eqnarray*}
    {\rm Det}_{2k+1}(V)=VF(V)\widehat{VF(V)},
    \end{eqnarray*}
    \begin{eqnarray}
    F(V):=\Phi^{-1}\circ f \circ \Phi(V),\label{thm:gen2}\quad
    \Phi:\cl_{p,q}\xrightarrow{\sim}
        \begin{cases}
            \mathbb{C}\otimes\cl_{p',q'},\quad \text{if $q'$ is even},\nonumber\\
            \mathbb{S}\otimes\cl_{p',q'},\quad \text{else}.
        \end{cases}  \nonumber
    \end{eqnarray}
\end{thm}
\begin{proof}
    Since by definition $F(V) = \Phi^{-1} \circ f \circ \Phi(V)$, we can write
    \[
        f \circ \Phi(V) = \Phi \circ F(V).
    \]
    Substituting this into \eqref{thm:gen1}, we obtain
    \begin{eqnarray*}
        \mathrm{Det}_{2k+1}(V)
        &=&\Phi^{-1}(\Phi(V)\Phi\circ F(V))\widehat{\Phi^{-1}}(\Phi(V)\Phi\circ F(V))\\
        &=&\Phi^{-1}\circ \Phi(VF(V))\widehat{\Phi^{-1}\circ\Phi}(VF(V))\\
        &=&VF(V)\widehat{VF(V)}
    \end{eqnarray*}
    since $\Phi^{-1} \circ \Phi$ acts as the identity on $\cl_{p,q}$.  
    This completes the proof.
\end{proof}

\begin{corollary}
    If ${\rm Det}_{2k+1}(V)\neq 0$, then the inverse of $V\in\cl_{p,q}, p+q=2k+1$ is given by
    $$
       V^{-1} = \frac{F(V)\widehat{VF(V)}}{{\rm Det}_{2k+1}(V)},
    $$
    where ${\rm Det}_{2k+1}$ is defined in Theorem \ref{thm:gen2}.
\end{corollary}

\begin{lemma}\label{lem:B}
    $\forall  S=1,2,\ldots,n-1, \exists \underline{S}\subseteq \mathbb{B}_{p,q}$ such that
    $$
    m_S(\Phi(V)) = \Phi(b_{\underline{S}}(V)),\quad  V\in\cl_{p,q}, p+q=n=2k+1. 
    $$
\end{lemma}

Note that $F(V)$ can indeed be computed without directly applying the isomorphism~$\Phi$.
Using Lemma~\ref{lem:B}, any conjugations of $\Phi(V)$ can be replaced by the corresponding basis conjugations of $V$.  
Therefore, the formula $F(V) = \Phi^{-1} \circ f \circ \Phi(V)$
is equivalent to applying $f$ directly to $V$, but replacing each conjugation by the respective basis conjugation.

\section{Conclusions}

In this paper, we have successfully derived an explicit formula for the determinant in geometric algebras $\cl_{p,q}$ with $p+q=7$. By introducing a novel family of basis conjugation operations, we have derived formulas that are intrinsic to the algebra, eliminating the need for external algebraic isomorphisms that have complicated previous approaches. This advancement yields a clearer geometric interpretation and enhances the stability and reliability of practical computations.

Our explicit formulas extend and refine earlier results \cite{Hitzer2}. Our main contributions are as follows. First, we presented an explicit, isomorphism-free determinant and inverse formulas for $\cl_{p,q}$ with 
$p+q=7$ (Theorems~\ref{n==7} and \ref{n===7}). Second, we developed and benchmarked a numerical implementation, demonstrating that our isomorphism-free formulas offer superior numerical robustness compared to earlier isomorphism-dependent methods, though with a modest increase in computational cost. Benchmarks also show that our isomorphism-free formulas have lower computational cost that known recursive method based on the Faddeev–LeVerrier algorithm \cite{Prodanov, Shirokov}. The implementation is publicly available at the repository  \url{github.com/kamranuz/clifford_7d}. Third, we established a general method (Theorem~\ref{thm:gen2}) for constructing isomorphism-free determinant and inverse formulas for Clifford algebras over any odd-dimensional vector space from known formulas for the preceding even dimension. 

The framework of basis conjugations developed here not only solves a specific problem in dimension seven but also provides a new algebraic perspective that could be applied to other invariant theoretic problems in geometric algebra. Future research directions include discovering analogous explicit determinant formulas in higher odd dimensions, optimizing the computational efficiency of the formulas further and exploring applications of these stable numerical methods in areas such as physics and engineering (see, for example, various applications of $\cl_{p,q}$ with $p+q=7$ in \cite{Trayling,Shaw,Ricardo}).

\section*{Acknowledgment}

The results of this paper were reported at the ENGAGE Workshop (Hong Kong, July 2025) within the International Conference Computer
Graphics International 2025 (CGI 2025). The authors are grateful to the organizers and the participants of this conference for fruitful discussions.

The authors are grateful to the anonymous Reviewers for their careful reading of the paper and helpful comments on how to improve the presentation.

The publication was prepared within the framework of the Academic Fund Program at HSE University (grant №26-00-017 Clifford algebras and matrix methods: theory and applications).

\textbf{\!Data availability.} \!Data sharing not applicable to this article as no datasets were generated or analyzed during the current study.

\textbf{\!Declarations conflict of interest.} The authors declare that they have no conflict of interest.

\end{document}